\documentclass[11pt]{amsart}

\usepackage[margin=1.15in]{geometry}
\usepackage{amsmath,amssymb,amsthm}
\usepackage{xcolor}
\usepackage[colorlinks=true,linkcolor=blue,citecolor=blue,urlcolor=blue]{hyperref}

\newtheorem{theorem}{Theorem}[section]
\newtheorem{lemma}[theorem]{Lemma}
\newtheorem{proposition}[theorem]{Proposition}

\newcommand{\R}{\mathbb R}
\newcommand{\dd}{\,d}
\newcommand{\loc}{\mathrm{loc}}

\title[Optimal regularity at the free boundary in mean field games]{Optimal regularity at the free boundary in one-dimensional first-order mean field games}
\author{Sebastian Munoz}
\address{Department of Mathematics, University of California, Los Angeles}
\email{sebastian@math.ucla.edu}
\date{}

\begin{document}

\begin{abstract}
We establish sharp regularity for the value function, the pressure, and the
free boundary in one-dimensional first-order mean field games with power
coupling and compactly supported density. Under a standard nondegeneracy assumption on the initial datum, the pressure
\(p=m^\theta\) is Lipschitz continuous, the value function \(u\) is
\(C^{1,1/2}\), and the two free boundary curves are smooth in time. If the
initial pressure is smooth, then both
\(p\) and \(u\) are smooth up to the free boundary from inside the positive
phase. The proof works in Lagrangian coordinates and, through a singular change of
variables, recasts the boundary degeneracy as a removable radial axis in effective dimension \(N=4+2/\theta\), allowing the application of recent estimates for even solutions to elliptic problems with degenerate weights.
\end{abstract}

\maketitle

\section{Introduction}

For one-dimensional first-order mean field games (MFG) with compactly supported density, we obtain the optimal regularity at the free boundary, settling the questions left open in \cite{CMP}.

We consider the system
\begin{equation}
\label{eq:mfg-system}
\begin{cases}
        -u_t+\frac12 u_x^2=m^\theta
                &\text{in }\R\times(0,T),\\
        m_t-(mu_x)_x=0
                &\text{in }\R\times(0,T),\\
        m(\cdot,0)=m_0,
\end{cases}
\end{equation}
where \(\theta>0\). We impose one of two terminal conditions. In the
terminal-cost case,
\begin{equation}
\label{ass:mfg-terminal}
        u(\cdot,T)=g(m(\cdot,T)),\qquad
        g(s)=c_1s^\theta,\qquad c_1\ge0 .
\end{equation}
In the planning case,
\begin{equation}
\label{ass:planning-terminal}
        m(\cdot,T)=m_T .
\end{equation}
The unknown \(m\) is the density of agents and \(u\) is the value function.
The system was introduced by Lasry and Lions in the theory of MFGs
\cite{LL06,LL07}. The planning version is also the Euler--Lagrange system for
a dynamic optimal transport problem with congestion, in the spirit of the
Benamou--Brenier formulation \cite{BB00}; see, for instance,
\cite{GMS19,LS18,OPS19}.

Regularity of \eqref{eq:mfg-system} in the everywhere-positive regime has been
developed through a degenerate elliptic reformulation due to Lions
\cite{Lions}; see \cite{Munoz22,MSM24,Porretta23}. The present paper concerns
the complementary regime, in which the density is compactly supported and a
free boundary appears. We refer to \cite{CP20} for a general introduction to
MFGs.

Write
\[
        m_0\ge0,\qquad \{m_0>0\}=(a,b),\qquad p_0:=m_0^\theta,\qquad
        p:=m^\theta.
\]
The Lagrangian flow is the characteristic map from the initial coordinate
\(y\in(a,b)\),
\[
        \gamma(y,0)=y,\qquad
        \gamma_t(y,t)=-u_x(\gamma(y,t),t).
\]
Under the assumptions stated below, \cite{CMP} shows that solutions have
finite speed of propagation and that the positive set is transported by
\(\gamma\):
\[
        \{m>0\}
        =\{(x,t):\gamma_L(t)<x<\gamma_R(t)\},\qquad
        \gamma_L(t):=\gamma(a,t),\quad
        \gamma_R(t):=\gamma(b,t).
\]
Moreover, \(\gamma_L\) and \(\gamma_R\) are \(C^{1,1}\), are strictly convex
toward the positive phase, and \(p\), \(u_x\), \(u_t\) are H\"older
continuous up to the free boundary.

The natural benchmark for sharpness is the explicit
self-similar solution of \eqref{eq:mfg-system} whose density has the form
\begin{equation*}
        m(x,t)=t^{-\nu}
        \left(R-\frac{\nu(1-\nu)}2\frac{x^2}{t^{2\nu}}
        \right)_+^{1/\theta},
        \qquad \nu=\frac2{2+\theta}.
\end{equation*}
This is the MFG analogue of the Barenblatt profile of the porous medium
equation (PME) \cite{Barenblatt,Vazquez}. For this solution the pressure
\(p=m^\theta\) vanishes linearly at the edge of the support, so its zero
extension is Lipschitz but not \(C^1\) across the free boundary, no matter how
smooth the initial datum. Consequently, the Eulerian pressure cannot be more
regular than \(W^{1,\infty}\) across the free boundary, and the computations in
\cite[App.~A]{CMP} show that the corresponding value function for this solution
is \(C^{1,1/2}\) but not \(C^{1,\gamma}\) across the free boundary for any
\(\gamma>1/2\).

Theorem
\ref{thm:higher-boundary-regularity} below shows that the space-time
Lipschitz bound for \(p\) and the \(C^{1,1/2}\) bound for \(u\) are sharp
and always attained. The free
boundary curves are also smooth in time, under the basic hypotheses on \(p_0\).
Moreover, the one-sided regularity of \(p\) from
the positive phase matches the regularity of the initial pressure \(p_0\),
while the value function gains one derivative.

The analogy with the PME extends beyond the self-similar profile: the
same linear vanishing of the pressure and the same role of a nondegeneracy
assumption on the initial pressure drive the classical free boundary
theory of the PME (see, e.g., \cite{CF79,DH98,KKV18,Vazquez}). Theorem
\ref{thm:higher-boundary-regularity} can be viewed as the MFG
counterpart of those results, in which the parabolic dissipation of the
PME is replaced by the Hamilton--Jacobi/continuity structure of
\eqref{eq:mfg-system}.

Let \(m_0\) have unit mass. We assume
\begin{equation}
\label{ass:initial-data}
        p_0\in C([a,b])\cap C^{1,\sigma}(a,b),\qquad
        p_0>0\text{ in }(a,b),\qquad
        p_0(a)=p_0(b)=0,
\end{equation}
for some \(\sigma\in(0,1)\), and, for some \(C_0,K_0,\delta>0\),
\begin{equation}
\label{ass:nondegenerate}
        C_0^{-1}\operatorname{dist}(x,\{a,b\})
        \le p_0(x)\le
        C_0\operatorname{dist}(x,\{a,b\}),
        \qquad x\in(a,b),
\end{equation}
\begin{equation}
\label{ass:h-second}
        p_0''\ge -K_0\text{ in }(a,b),\qquad
        p_0''\le0\text{ in }(a,a+\delta)\cup(b-\delta,b),
\end{equation}
in the distributional sense. In the planning problem we further assume
\begin{equation}
\label{ass:planning-basic}
        m_T\ge0,\qquad \{m_T>0\}=(a_1,b_1),\qquad
        p_T:=m_T^\theta\in
        C([a_1,b_1])\cap C^{1,\sigma}(a_1,b_1),
\end{equation}
that \(m_T\) has unit mass, and
\begin{equation}
\label{ass:terminal-data}
        C_0^{-1}\operatorname{dist}(x,\{a_1,b_1\})
        \le p_T(x)\le
        C_0\operatorname{dist}(x,\{a_1,b_1\}),
        \qquad x\in(a_1,b_1).
\end{equation}
Throughout, by a solution \((u,m)\) to \eqref{eq:mfg-system} we mean a pair
such that \(u\in C^1(\R\times(0,T))\) with \(u_x\in L^\infty(\R\times(0,T))\)
satisfies the Hamilton--Jacobi equation pointwise, \(m\in C_b(\R\times[0,T])\)
satisfies the continuity equation in the distributional sense, and the
prescribed initial/terminal conditions hold. By
\cite[Thm.~4.3]{CMP}, the solution is unique in the terminal-cost case;
in the planning case, \(m\) is uniquely determined while \(u\) is unique
on \(\{m>0\}\) up to an additive constant. The regularity statements
below hold for all such solutions.

\begin{theorem}[Regularity of the value function, pressure, and free boundary]
\label{thm:higher-boundary-regularity}
Let \((u,m)\) be a solution of \eqref{eq:mfg-system} with either
\eqref{ass:mfg-terminal} or \eqref{ass:planning-terminal}. Assume
\eqref{ass:initial-data}--\eqref{ass:h-second}, and also
\eqref{ass:planning-basic}, \eqref{ass:terminal-data} in the planning case.
Let \(\gamma\) be the Lagrangian flow and set
\(\gamma_L:=\gamma(a,\cdot)\), \(\gamma_R:=\gamma(b,\cdot)\), so that
\(\{m>0\}=\{(x,t):\gamma_L(t)<x<\gamma_R(t)\}\). Set \(p:=m^\theta\). Then,
for every compact interval \(J\subset(0,T)\),
\[
\begin{gathered}
        \gamma_L,\gamma_R\in C^\infty_{\loc}(0,T),\\
        p\in W^{1,\infty}_{\loc}(\R\times(0,T))
          \cap C^{1,\sigma}\bigl(\overline{\{m>0\}}\cap(\R\times J)\bigr),\\
        u\in C^{1,1/2}_{\loc}(\R\times(0,T))
          \cap C^{2,\sigma}\bigl(\overline{\{m>0\}}\cap(\R\times J)\bigr).
\end{gathered}
\]
If, in addition, \(p_0\in C^{\ell,\sigma}([a,b])\) for some integer
\(\ell\ge2\), then, for every compact interval
\(J\subset(0,T)\),
\[
        p\in
        C^{\ell,\sigma}\bigl(\overline{\{m>0\}}\cap(\R\times J)\bigr),
        \qquad
        u\in
        C^{\ell+1,\sigma}\bigl(\overline{\{m>0\}}\cap(\R\times J)\bigr).
\]
In particular, if \(p_0\) is smooth on \([a,b]\), then \(p\) and \(u\) are
smooth up to the free boundary from the positive phase.
\end{theorem}

In the planning case, the higher-regularity hypothesis is placed on
\(p_0\); by the time reversal \((u,m)(x,t)\mapsto(-u(x,T-t),m(x,T-t))\),
which preserves \eqref{eq:mfg-system} and exchanges initial and terminal
data, the analogous conclusion holds with the burden shifted to \(p_T\).

By translation, we may and do assume \(a=0\) throughout the
proof, so the support of \(m_0\) is \((0,b)\); the analysis at the right
endpoint \(b\) is symmetric.
The proof works in Lagrangian coordinates, in which the free boundary is
fixed at \(\{0,b\}\). With \(y\) the Lagrangian variable, the flow
\(\gamma\) satisfies the space-time degenerate elliptic equation
\begin{equation}
\label{eq:intro-flow}
        \gamma_{tt}
        +\frac{\theta p_0}{\gamma_y^{2+\theta}}\gamma_{yy}
        =
        \frac{p_0'}{\gamma_y^{1+\theta}} .
\end{equation}
In terms of
\[
        Z:=\gamma_y^{-(\theta+1)},
\]
this becomes
\[
        \gamma_{tt}=p_0'Z+\frac{\theta}{\theta+1}p_0Z_y.
\]
Averaging this identity over a one-sided neighborhood of each endpoint, with the
(possibly singular) term \(p_0Z_y\) integrated by parts, formally yields the
acceleration formulas
\[
        \gamma_L''=p_0'(0^+)Z(0,\cdot),
        \qquad
        \gamma_R''=p_0'(b^-)Z(b,\cdot).
\]
The averaging becomes rigorous as soon as \(Z\) admits traces at the
endpoints, and the regularity of \(\gamma_L,\gamma_R\) is then controlled by
the boundary regularity of \(Z\).

The degeneracy in \eqref{eq:intro-flow} comes from the linear vanishing of
\(p_0\) at the endpoints. The first step is to absorb this degeneracy
into a radial weight in a higher effective dimension. Near \(0\), introduce the
square-root distance coordinate
\[
        y=\frac{r^2}{4}.
\]
The linear vanishing \(p_0(y)\simeq y\) becomes quadratic in \(r\), while
\(\partial_y=(2/r)\partial_r\). These two effects combine to convert the
endpoint degeneracy into a radial weight, rather than a degenerating
ellipticity constant. Setting \(\widetilde Z(r,t)=Z(r^2/4,t)\), one obtains, away from
\(r=0\), the divergence-form equation
\[
        \partial_t\!\left(W(r)\beta(\widetilde Z)\widetilde Z_t\right)
        +\partial_r\!\left(W(r)A(r)\widetilde Z_r\right)
        +W(r)D(r)\widetilde Z=0,
\]
where \(W(r)\simeq r^{N-1}\) with effective dimension
\begin{equation}
\label{eq:N-intro}
        N:=4+\frac2\theta,
\end{equation}
the coefficients \(A\)
and \(\beta(\widetilde Z)\) are uniformly positive and bounded, and \(D\in L^\infty\).
Since \(W\) matches the radial Jacobian in dimension \(N\), the endpoint
\(r=0\) plays the role of a radial axis. The condition \(N>2\) is crucial: the radial
axis then has zero capacity in the weighted Sobolev space, in the sense that
cutoffs vanishing near \(r=0\) have weighted energy that disappears as the
cutoff is removed, and the weak formulation, originally obtained for test
functions supported away from the axis, extends across \(r=0\) without
imposing any additional boundary condition. In this sense the axis is
removable: \(\widetilde Z\) extends by even reflection across \(r=0\) to a weak
solution of the same divergence-form equation on the full ball, with no
compatibility imposed at the axis. The original free boundary question
thus becomes one of interior regularity for an even solution of the
weighted equation. A degenerate
H\"older estimate \cite{Zamboni}, followed by a recent weighted
Schauder estimate for even solutions \cite{STV2021}, gives
\(C^{1,\tau}\) control up to the axis, for every
\(\tau\in(0,1)\).
Since the coefficients depending on \(p_0\) are independent of \(t\),
the Schauder estimate can be iterated on tangential \(t\)-derivatives.
Under the basic endpoint hypotheses,
this gives \(C^\infty\) regularity in \(t\) for the endpoint trace
\(Z(0,\cdot)\), hence for the boundary acceleration, and hence
\(\gamma_L,\gamma_R\in C^\infty\).

The remaining normal regularity comes from the regular-singular form of
the Lagrangian equation in the original variable \(y\). Writing
\(p_0(y)=y\,h(y)\) with \(h(0)>0\), the non-divergence form of
\eqref{eq:Z-div} carries a single power of \(y\) in front of the highest
\(y\)-derivative; setting \(V:=Z_y\) and dividing by \(c_\theta\,h(y)\),
the equation for \(V\) takes the regular-singular form
\[
        y V_y+b(y)V=F,\qquad
        b(0)=\frac{1+c_\theta}{c_\theta}>1 .
\]
The homogeneous singular solution \(y^{-b(0)}\) of this ODE is
non-integrable at \(0\), so a contribution of this mode in \(V\) would
force its primitive to be unbounded near \(0\); since \(Z\) is bounded,
the singular mode is absent and \(V\) is given by the Volterra integral
of the regular branch. Bounded solutions of the equation therefore
inherit the regularity of \(F\). When \(p_0\) is one-sided
\(C^{2,\sigma}\), this already supplies the first missing normal
derivative; under higher one-sided regularity of \(p_0\), differentiating
the equation in \(y\) and \(t\) and iterating yields successive normal
derivatives, each step requiring one more derivative of \(p_0\). At the
top order, only the weighted derivative \(y\,\partial_y^{n+1}V\) is
controlled, but this weight is precisely absorbed by the vanishing
factor of \(p_0(y)=y\,h(y)\) in the quantities to be transferred to
Eulerian coordinates, such as \(\gamma(y,t)-\gamma_L(t)\) and
\(p(\gamma(y,t),t)=p_0(y)\,Z(y,t)^{\theta/(\theta+1)}\).

The endpoint estimates from the regular-singular descent then transfer to
the pressure and the value function via standard formulas. The mass relation
\begin{equation}
\label{eq:mass-relation}
        p(\gamma(y,t),t)=p_0(y)\gamma_y(y,t)^{-\theta}
\end{equation}
gives the regularity of the pressure in Lagrangian coordinates. The
identity \(u_x\circ\gamma=-\gamma_t\), together with the extra tangential
regularity in the endpoint charts, gives one more derivative for the value
function. Finally, the change of variables
\(r\mapsto y\) produces one-sided endpoint charts for the flow with nonzero
one-sided derivative at the boundary; inverting these charts and composing
with the Lagrangian pressure and Lagrangian velocity transfers the estimates
to Eulerian coordinates, and the Hamilton--Jacobi equation then yields the
corresponding regularity of \(u\) itself.

The paper is organized as follows. Section \ref{sec:preliminaries} fixes the
Lagrangian formulation. Section \ref{sec:boundary-equation} derives the
weighted divergence equation \eqref{eq:radial-weak} via the square-root
substitution. Section \ref{sec:endpoint-regularity} establishes the endpoint
regularity, combining the radial-axis H\"older/Schauder bootstrap with a
regular-singular descent in Lagrangian coordinates. Section
\ref{sec:proof-main-theorem} proves Theorem
\ref{thm:higher-boundary-regularity}.

\section{Preliminaries and Lagrangian formulation}
\label{sec:preliminaries}

We work under the assumptions of Theorem
\ref{thm:higher-boundary-regularity}, with \(a=0\) by the convention
introduced above. From now on \(y\in(0,b)\) denotes the Lagrangian variable,
while \(x\) denotes the Eulerian space variable; derivatives of \(p_0\) with
respect to \(y\) are written \(p_0',p_0''\).

\begin{lemma}
\label{lem:h-boundary}
After decreasing \(\delta\), if necessary,
\begin{equation*}
        p_0\in C^{1,1}([0,\delta))\cap C^{1,1}((b-\delta,b]),
        \qquad p_0'(0^+)>0,\qquad p_0'(b^-)<0 .
\end{equation*}
\end{lemma}

\begin{proof}
On \((0,\delta)\) and \((b-\delta,b)\), \eqref{ass:h-second} gives
\(-K_0\le p_0''\le0\) distributionally. Thus \(p_0''\) is represented by an
\(L^\infty\) function on these one-sided neighborhoods, and \(p_0'\) is Lipschitz there.
The one-sided derivatives at \(0,b\) therefore exist. If \(p_0'(0^+)=0\), then
the Lipschitz bound on \(p_0'\) gives \(p_0(y)\le Cy^2\), contrary to
\eqref{ass:nondegenerate}. The same argument excludes \(p_0'(b^-)=0\), and
\eqref{ass:initial-data} fixes the signs.
\end{proof}

For the analysis at the left endpoint we set
\begin{equation}
\label{eq:h-def}
        h(y):=\frac{p_0(y)}{y}\quad(y\in(0,\delta)),\qquad
        h(0):=p_0'(0^+),
\end{equation}
so that \(p_0(y)=y\,h(y)\) on \([0,\delta)\). By
Lemma~\ref{lem:h-boundary}, \(h\in C^{0,1}([0,\delta))\) and \(h(0)>0\); if
in addition \(p_0\in C^{\ell,\sigma}([0,\delta])\) for some integer
\(\ell\ge2\) and some \(\sigma\in(0,1)\), then
\(h\in C^{\ell-1,\sigma}([0,\delta))\). The analogous function at \(b\) is
defined symmetrically, after replacing \(y\) by \(b-y\).

We use the following consequences of
\cite[Thms.~4.3 and~4.10, Cor.~4.11, Prop.~4.13]{CMP}.

\begin{theorem}[Free boundary estimates]
\label{thm:free-boundary-estimates}
Under \eqref{ass:initial-data}--\eqref{ass:h-second}, together with
\eqref{ass:mfg-terminal} in the terminal cost case and
\eqref{ass:planning-terminal}, \eqref{ass:planning-basic},
\eqref{ass:terminal-data} in the planning case, the
Lagrangian flow
\(\gamma:(0,b)\times(0,T)\to\R\) extends continuously to
\([0,b]\times(0,T)\). Moreover, there exist \(c,C>0\) such that
\[
        c\le \gamma_y(y,t)\le C
\]
for \((y,t)\in(0,b)\times[0,T]\);
\(\gamma\) is \(C^2_{\loc}((0,b)\times(0,T))\) and solves \eqref{eq:flow}
classically in \((0,b)\times(0,T)\).
Here \(c,C\) depend on \(T,\ T^{-1},\ \theta,\ \theta^{-1},\ C_0,\ K_0,\ \delta^{-1}\),
together with \(c_1\) in the terminal-cost case.
\end{theorem}

Let \(\gamma\) be the Lagrangian flow. By Theorem
\ref{thm:free-boundary-estimates},
\begin{equation}
\label{eq:qy-bounds}
        0<c\le \gamma_y(y,t)\le C
\end{equation}
for \((y,t)\in(0,b)\times[0,T]\). Put \(\gamma_L:=\gamma(0,\cdot)\) and
\(\gamma_R:=\gamma(b,\cdot)\). Moreover \(\gamma\) solves
\begin{equation}
\label{eq:flow}
        \gamma_{tt}
        +\frac{\theta p_0}{\gamma_y^{2+\theta}}\gamma_{yy}
        =
        \frac{p_0'}{\gamma_y^{1+\theta}}
\end{equation}
classically in \((0,b)\times(0,T)\).

\section{The boundary equation}
\label{sec:boundary-equation}

Put
\begin{equation}
\label{eq:Z-def}
        Z:=\gamma_y^{-(\theta+1)},\qquad
        \beta(Z):=\frac1{\theta+1}Z^{-1-\frac1{\theta+1}},\qquad
        c_\theta:=\frac{\theta}{\theta+1} .
\end{equation}
Fix \(I\Subset(0,T)\). By \eqref{eq:qy-bounds}, \(\gamma_y\)---and hence
\(Z\) and \(\beta(Z)\)---is bounded above and below by positive constants
on the whole of \((0,b)\times[0,T]\). Since
\[
        Z_y=-(\theta+1)\gamma_y^{-(\theta+2)}\gamma_{yy},
\]
equation \eqref{eq:flow} is equivalent to
\begin{equation}
\label{eq:gamma-Z}
      \gamma_{tt}=p_0'Z+c_\theta p_0Z_y.
\end{equation}
By Theorem \ref{thm:free-boundary-estimates},
\(\gamma\in C^2((0,b)\times(0,T))\), so \(Z\in C^1((0,b)\times(0,T))\) and
the identity \(\gamma_{yt}=-\beta(Z)Z_t\) holds pointwise in
\((0,b)\times(0,T)\). Differentiating \eqref{eq:gamma-Z} in \(y\) and
commuting mixed derivatives gives
\begin{equation}
\label{eq:Z-div}
        \partial_t\!\left(\beta(Z)Z_t\right)
        +\partial_y\!\left(p_0'Z+c_\theta p_0Z_y\right)=0
\end{equation}
in the weak sense on \((0,b)\times I\).

We now localize near \(0\) and recast \eqref{eq:Z-div} in weighted divergence
form. The weight is chosen so that the first-order term \(p_0'Z\) is absorbed
into a divergence and so that, after the square-root substitution below, it
becomes the radial Jacobian in dimension \(N\). Multiply \eqref{eq:Z-div} by
\(p_0^{1/c_\theta}\). Throughout the one-sided neighborhood \((0,\delta)\times I\),
\(p_0>0\) and
\((p_0^{1/c_\theta})_y/p_0^{1/c_\theta}=c_\theta^{-1}p_0'/p_0\), giving, in
the weak sense, the divergence form equation
\begin{equation}
\label{eq:Z-weighted-y}
        \partial_t\!\left(p_0^{1/c_\theta}\beta(Z)Z_t\right)
        +\partial_y\!\left(c_\theta p_0^{1+1/c_\theta}Z_y\right)
        +p_0^{1/c_\theta}p_0''Z=0 .
\end{equation}
Let \(r_0>0\) be small. Set
\[
        y=\frac{r^2}{4},\qquad \widetilde Z(r,t):=Z\!\left(\frac{r^2}{4},t\right).
\]
Then \eqref{eq:Z-weighted-y} takes the form
\begin{equation}
\label{eq:radial-weak}
        \iint W(r)\Bigl[
        \beta(\widetilde Z)\widetilde Z_t\Phi_t+A(r)\widetilde Z_r\Phi_r-D(r)\widetilde Z\Phi
        \Bigr]\dd r\dd t=0 ,
\end{equation}
for every \(\Phi\in C_c^\infty((0,r_0)\times I)\), where
\begin{equation*}
        W(r):=\frac r2p_0\!\left(\frac{r^2}{4}\right)^{1/c_\theta},\qquad
        A(r):=\frac{4c_\theta p_0\!\left(\frac{r^2}{4}\right)}{r^2},\qquad
        D(r):=p_0''\!\left(\frac{r^2}{4}\right).
\end{equation*}
Moreover, \(\widetilde Z\in H^1_{\loc}((0,r_0)\times I)\); the
Caccioppoli step in the proof of Lemma~\ref{lem:radial-regularity} below
upgrades this to the weighted finite-energy bound
\(\widetilde Z\in H^1_{\loc}([0,r_0)\times I;\dd\mathfrak m)\) up to the
axis. By Lemma \ref{lem:h-boundary}, after decreasing \(r_0\) if
necessary,
\begin{equation}
\label{eq:radial-structure}
        W(r)\simeq r^{N-1},\qquad
        N:=4+\frac2\theta>2,\qquad
        0<c\le A(r)\le C,\qquad
        D\in L^\infty .
\end{equation}
The same reduction applies at \(b\), after replacing \(y\) by \(b-y\).

\section{Endpoint regularity}
\label{sec:endpoint-regularity}

In the absence of lower-order terms, H\"older regularity for
divergence-form elliptic equations with degenerate weights goes back to
the classical work of Fabes, Kenig, and Serapioni~\cite{FKS82}. We use the following
extension that allows a bounded zero-order coefficient, a consequence of
\cite[Thm.~5.2 and Lem.~5.1]{Zamboni}.

\begin{theorem}[Power-weight H\"older estimate]
\label{thm:zamboni}
Let \(\Omega\subset\R^{n+1}\), write \(z=(z',s)\in\R^n\times\R\), and let
\(-n<\tau<n\). Suppose that \(\omega\) is comparable on compact subsets of
\(\Omega\) to \(|z'|^\tau\), and let \(\mathcal A\) be measurable and symmetric
with
\[
        \lambda \omega|\xi|^2
        \le \mathcal A(z)\xi\cdot\xi
        \le \Lambda\omega|\xi|^2 .
\]
Let \(d\in L^\infty_{\loc}(\Omega)\). Every
\(u\in H^1_{\loc}(\Omega;\omega)\) satisfying
\[
        \int_\Omega \mathcal A\nabla u\cdot\nabla\phi
        +\int_\Omega \omega d u\phi=0,\qquad \phi\in C_c^\infty(\Omega),
\]
is locally H\"older continuous. The exponent and local estimate depend only on
\(n,\tau,\lambda,\Lambda\), the local comparability constants for \(\omega\),
the local \(L^\infty\) bound for \(d\), the compact subdomain, and a local size
of \(u\).
\end{theorem}

The next lemma realizes the zero-capacity removability of the axis, together with the resulting H\"older
regularity up to it.

\begin{lemma}[Removability of the radial axis and boundary H\"older estimate]
\label{lem:radial-regularity}
Let \(I\) be an open interval, \(r_0>0\), and \(N>2\). Let
\begin{equation}
\label{ass:radial-weight}
        \dd\mathfrak m=W(r)\dd r\dd t,\qquad W(r)\simeq r^{N-1}
\end{equation}
on \((0,r_0)\). Suppose that
\(\widetilde Z\in L^\infty((0,r_0)\times I)\cap H^1_{\loc}((0,r_0)\times I)\) satisfies
\begin{equation}
\label{eq:radial-model}
        \iint W(r)\Bigl[
        B(r,t)\widetilde Z_t\Phi_t+A(r)\widetilde Z_r\Phi_r-D(r)\widetilde Z\Phi
        \Bigr]\dd r\dd t=0,
        \qquad \Phi\in C_c^\infty((0,r_0)\times I),
\end{equation}
where \(A,B,D\) are measurable and
\begin{equation}
\label{ass:radial-coefficients}
        0<\kappa\le A(r)\le K,\qquad
        0<\kappa\le B(r,t)\le K,
        \qquad D\in L^\infty((0,r_0)).
\end{equation}
Then \(\widetilde Z\) has a representative in
\(C^\alpha_{\loc}([0,r_0)\times I)\cap
H^1_{\loc}([0,r_0)\times I;\dd\mathfrak m)\), for some \(\alpha\in(0,1)\)
depending only on \(N,\kappa,K\), \(\|D\|_\infty\), and the comparability
constants in~\eqref{ass:radial-weight}; the local H\"older norm also depends
on \(\|\widetilde Z\|_\infty\) and the cylinder. Moreover, \eqref{eq:radial-model}
remains valid for test functions compactly supported in the relative
topology of \([0,r_0)\times I\).
\end{lemma}

\begin{proof}
Write \(M:=\|\widetilde Z\|_{L^\infty}\), \(L:=\|D\|_{L^\infty}\), and
\(\nabla=(\partial_r,\partial_t)\). For \(z_0=(r_1,t_1)\) set
\[
        Q_\rho^+(z_0):=([0,r_0)\times I)
        \cap\{|r-r_1|<\rho,\ |t-t_1|<\rho\}.
\]
It suffices to consider \(Q_{2R}^+(z_0)\Subset[0,r_0)\times I\).
Choose
\(\eta\in C_c^\infty(Q_{2R}^+(z_0))\) and
\(\chi_\varepsilon\) such that
\[
        \chi_\varepsilon=0\text{ on }(0,\varepsilon),\qquad
        \chi_\varepsilon=1\text{ on }(2\varepsilon,r_0),\qquad
        |\chi_\varepsilon'|\le C\varepsilon^{-1}.
\]
The estimate
\begin{equation}
\label{eq:axis-capacity}
        \int_\varepsilon^{2\varepsilon}W(r)|\chi_\varepsilon'(r)|^2\,dr
        \le C\varepsilon^{N-2}\to0
\end{equation}
uses only \(W\simeq r^{N-1}\) and \(N>2\).

Since \(\chi_\varepsilon\) vanishes near \(r=0\), the function
\(\widetilde Z\eta^2\chi_\varepsilon^2\) is an admissible test after smooth
approximation on \((\varepsilon,r_0)\times I\). Testing
\eqref{eq:radial-model}, using \eqref{ass:radial-coefficients}, Young's
inequality, and \(|\widetilde Z|\le M\), gives
\begin{align}
\label{eq:caccioppoli-eps}
        \iint W\eta^2\chi_\varepsilon^2|\nabla \widetilde Z|^2
        &\le
        C\iint W\widetilde Z^2\chi_\varepsilon^2|\nabla\eta|^2
        +C\iint W\widetilde Z^2\eta^2|\chi_\varepsilon'|^2       \notag\\
        &\quad
        +CL\iint W\widetilde Z^2\eta^2\chi_\varepsilon^2 .
\end{align}
The middle term tends to zero by \eqref{eq:axis-capacity}. Letting
\(\varepsilon\downarrow0\), lower semicontinuity and dominated convergence
yield
\begin{equation}
\label{eq:caccioppoli}
        \iint W\eta^2|\nabla \widetilde Z|^2
        \le
        C\iint W\widetilde Z^2|\nabla\eta|^2
        +CL\iint W\widetilde Z^2\eta^2 .
\end{equation}
Thus, after choosing \(\eta\equiv1\) on smaller relative cylinders and
covering compact subsets,
\begin{equation}
\label{eq:finite-energy-axis}
        \widetilde Z\in H^1_{\loc}([0,r_0)\times I;\dd\mathfrak m).
\end{equation}

We now use the same cutoff to remove the axis from the weak formulation.
Let \(\Phi\in C_c^\infty(Q_{2R}^+(z_0))\). Since
\(\Phi\chi_\varepsilon\in C_c^\infty((0,r_0)\times I)\), it is admissible in
\eqref{eq:radial-model}. Hence
\[
        \iint W\bigl[B\widetilde Z_t\Phi_t+A\widetilde Z_r\Phi_r-D\widetilde Z\Phi\bigr]\chi_\varepsilon
        =
        -\iint WA\widetilde Z_r\Phi\chi_\varepsilon' .
\]
By \eqref{eq:finite-energy-axis} and \eqref{eq:axis-capacity},
\[
        \left|\iint WA\widetilde Z_r\Phi\chi_\varepsilon'\right|
        \le C
        \biggl(\iint_{\varepsilon<r<2\varepsilon}W|\widetilde Z_r|^2\biggr)^{1/2}
        \biggl(\iint_{\varepsilon<r<2\varepsilon}
        W|\chi_\varepsilon'|^2\biggr)^{1/2}
        \to0 .
\]
The left-hand side converges by dominated convergence, using
\eqref{eq:finite-energy-axis} and the boundedness of \(\widetilde Z\). Therefore
\[
        \iint W\bigl[B\widetilde Z_t\Phi_t+A\widetilde Z_r\Phi_r-D\widetilde Z\Phi\bigr]=0,
\]
so \eqref{eq:radial-model} is valid for test functions compactly supported in
the relative topology of \([0,r_0)\times I\).

Let \(n:=\lceil N\rceil\) and \(\tau:=N-n\), so that \(-1<\tau\le0\). Set
\[
        \omega(\zeta):=\frac{W(|\zeta|)}{|\zeta|^{\,n-1}},
        \qquad \zeta\in B_{r_0}^n\setminus\{0\} .
\]
After defining \(\omega\) at \(\zeta=0\) arbitrarily, say by continuity when
available, we have \(\omega\simeq|\zeta|^\tau\) on compact subsets of
\(B_{r_0}^n\), as required by Theorem~\ref{thm:zamboni}.
Define
\[
        \widehat Z(\zeta,t):=\widetilde Z(|\zeta|,t),\qquad
        \mathcal A(\zeta,t):=\omega(\zeta)
        \begin{pmatrix}
        A(|\zeta|)I_n&0\\
        0&B(|\zeta|,t)
        \end{pmatrix},
        \qquad
        \mathfrak c(\zeta,t):=\omega(\zeta)D(|\zeta|).
\]
By \eqref{eq:finite-energy-axis},
\(\widehat Z\in H^1_{\loc}(B_{r_0}^n\times I;\omega)\). For
\(\Psi\in C_c^\infty(B_{r_0}^n\times I)\), let
\[
        \overline\Psi(r,t):=\frac1{|\mathbb S^{n-1}|}
        \int_{\mathbb S^{n-1}}\Psi(r\sigma,t)\dd\sigma .
\]
Since \(\overline\Psi\) is an admissible endpoint-touching test in
\eqref{eq:radial-model}, spherical integration gives
\[
        \iint \mathcal A\nabla\widehat Z\cdot\nabla\Psi
        -\iint \mathfrak c\,\widehat Z\Psi=0 .
\]
Theorem \ref{thm:zamboni}, applied with \(d=-D(|\zeta|)\), gives
\(\widehat Z\in C^\alpha_{\loc}(B_{r_0}^n\times I)\), after decreasing
\(\alpha\) if necessary. Restricting to radial lines gives the asserted
representative of \(\widetilde Z\).
\end{proof}

\subsection{Schauder bootstrap and tangential smoothness}

We now promote the H\"older estimate of Lemma~\ref{lem:radial-regularity}
to full Schauder regularity up to the radial axis, and extract the
smoothness of the trace \(\widetilde Z(0,\cdot)\) in the tangential variable~\(t\).
A key ingredient is the following corollary of the weighted Schauder estimates for
even solutions due to Sire--Terracini--Vita
\cite[Thms.~1.2 and 1.3]{STV2021}.

\begin{theorem}[Weighted Schauder estimate for even solutions]
\label{thm:stv-even-schauder}
Let \(n\ge1\), \(\mathfrak a>-1\), \(\alpha\in(0,1)\), and let \(B_1\) be
the unit ball in \(\R^n_x\times\R_y\). Let
\(a_1,\ldots,a_{n+1}\in C^{0,\alpha}(B_1)\) be even in \(y\), satisfy
\(\lambda\le a_i\le\Lambda\), and set
\[
        \mathcal A(x,y):=\operatorname{diag}
        (a_1(x,y),\ldots,a_{n+1}(x,y)).
\]
Let \(f\in L^q(B_1;|y|^{\mathfrak a})\), where
\begin{equation}
\label{eq:stv-q-condition}
        q>n+1+\mathfrak a^+,\qquad
        \alpha\le1-\frac{n+1+\mathfrak a^+}{q},
\end{equation}
and assume that \(f\) is even in \(y\). Let
\(\mathcal F\in C^{0,\alpha}(B_1;\R^{n+1})\), with
\(\mathcal F_i\) even in \(y\) for \(1\le i\le n\) and
\(\mathcal F_{n+1}\) odd in \(y\). If \(u\in H^1(B_1;|y|^{\mathfrak a})\)
is even in \(y\) and satisfies
\begin{equation}
\label{eq:stv-whole-weak}
        \int_{B_1}|y|^{\mathfrak a}
        \bigl(\mathcal A\nabla u+\mathcal F\bigr)\cdot\nabla\phi\dd x\dd y
        =\int_{B_1}|y|^{\mathfrak a}f\phi\dd x\dd y,
        \qquad \phi\in C_c^\infty(B_1),
\end{equation}
then \(u\in C^{1,\alpha}_{\loc}(B_1)\), with a local estimate depending only
on \(n,\mathfrak a,\alpha,q,\lambda,\Lambda\), the relevant H\"older and
\(L^q\) norms of the data, and
\(\|u\|_{L^2(B_1;|y|^{\mathfrak a})}\).
\end{theorem}

Write \(W(r)=r^{N-1}\omega_0(r)\) and, for \(r\ge0\),
\[
        \mathcal A_{\widetilde Z}(r,t):=\omega_0(r)
        \begin{pmatrix}
        A(r)&0\\
        0&\beta(\widetilde Z(r,t))
        \end{pmatrix},
        \qquad
        f_{\widetilde Z}(r,t):=\omega_0(r)D(r)\widetilde Z(r,t).
\]
Lemma \ref{lem:radial-regularity} extends the weak formulation
\eqref{eq:radial-weak} from test functions in \((0,r_0)\times I\) to test
functions touching \(r=0\). Hence, on each small half-ball in the
\((t,r)\)-plane, we may reflect
\[
        \widetilde Z^e(r,t):=\widetilde Z(|r|,t),\qquad
        \mathcal A_{\widetilde Z}^e(r,t):=\mathcal A_{\widetilde Z}(|r|,t),
        \qquad
        f_{\widetilde Z}^e(r,t):=f_{\widetilde Z}(|r|,t).
\]
Indeed, for \(\psi\in C_c^\infty\) in the reflected ball, set
\(\Phi(r,t):=\psi(r,t)+\psi(-r,t)\) for \(r\ge0\). Testing the
endpoint-touching form of \eqref{eq:radial-weak} with \(\Phi\) and changing
variables in the reflected term gives \eqref{eq:stv-whole-weak} for \(\widetilde Z^e\),
\(\mathcal A_{\widetilde Z}^e\), and \(f_{\widetilde Z}^e\); the reflected matrix is diagonal with
diagonal entries even in \(r\).
We are thus in the setting of Theorem~\ref{thm:stv-even-schauder} with
\(n=1\), the variables \((x_1,y)\) of the theorem playing the roles of
\((t,r)\) (so the weight \(|y|^{\mathfrak a}\) becomes \(|r|^{\mathfrak a}\),
and the diagonal entries of \(\mathcal A_{\widetilde Z}\) are read in the
permuted order); the weight exponent is \(\mathfrak a=N-1=3+2/\theta\) and
\(\mathcal F\equiv0\). Condition \eqref{eq:stv-q-condition} reduces to
\(q>N+1\) and \(\alpha\le1-(N+1)/q\).

\begin{proposition}[Schauder bootstrap and smoothness of the axis trace]
\label{prop:endpoint-schauder-bootstrap}
Under the assumptions of Theorem~\ref{thm:higher-boundary-regularity},
let \((u,m)\) be the corresponding solution, \(\gamma\) the associated
Lagrangian flow, \(Z\) as in \eqref{eq:Z-def},
\(\widetilde Z(r,t):=Z(r^2/4,t)\), and \(I\Subset(0,T)\). Then for every
\(\tau\in(0,1)\) and every \(j\ge0\),
\[
        \partial_t^j\widetilde Z\in C^{1,\tau}_{\loc}([0,r_0)\times I).
\]
In particular, \(\widetilde Z(0,\cdot)\in C^\infty_{\loc}(I)\).
\end{proposition}

\begin{proof}
We first record, at the formal level, the linear equation that drives the
tangential induction. Put \(\widetilde Z_0:=\widetilde Z\), \(\mathfrak F_0:=0\), and
\(f_0:=D\widetilde Z\), and, for \(j\ge1\), write \(\widetilde Z_j:=\partial_t^j\widetilde Z\);
that this derivative exists in the weighted energy class is part of the
inductive claim and is established by the difference-quotient step below.
Granting this, formal differentiation of \eqref{eq:radial-weak} \(j\) times
in \(t\) shows that \(\widetilde Z_j\) satisfies
\begin{equation}
\label{eq:bootstrap-linear}
        \partial_t\!\left(W\beta(\widetilde Z)(\widetilde Z_j)_t\right)
        +\partial_r\!\left(WA(\widetilde Z_j)_r\right)
        =-Wf_j+\partial_t(W\mathfrak F_j),
\end{equation}
where \(f_j:=D\widetilde Z_j\) and \(\mathfrak F_j\) is a finite linear combination of
products of \(\beta^{(k)}(\widetilde Z)\), \(1\le k\le j\), with the tangential
derivatives \(\partial_t^i\widetilde Z\), \(1\le i\le j\). Equivalently, for \(j\ge0\),
\begin{equation}
\label{eq:bootstrap-weak}
        \iint W\Bigl[
        \beta(\widetilde Z)(\widetilde Z_j)_t\Phi_t+A(\widetilde Z_j)_r\Phi_r
        -\mathfrak F_j\Phi_t-f_j\Phi
        \Bigr]\dd r\dd t=0
\end{equation}
for every \(\Phi\in C_c^\infty([0,r_0)\times I)\), where \([0,r_0)\) is
taken with its relative topology so that \(\Phi\) may touch \(r=0\). The principal coefficient is
\(\mathcal A_{\widetilde Z}\) as in the base case. In the weak form associated with
Theorem~\ref{thm:stv-even-schauder}, the lower-order data are the bulk
source \(\omega_0 f_j\) and the tangential field
\(\mathcal F=(0,-\omega_0 \mathfrak F_j)\) in the \((r,t)\)-components.

\emph{Difference-quotient step.} We now justify \eqref{eq:bootstrap-weak}
inductively in \(j\). Let \(\delta_h\) denote the difference quotient in
\(t\), and suppose that, for some \(\alpha\in(0,1)\),
\(\widetilde Z_{j-1}\in C^{1,\alpha}_{\loc}\cap H^1_{\loc}([0,r_0)\times I;\dd\mathfrak m)\)
and that \eqref{eq:bootstrap-weak} holds with index \(j-1\). On cylinders
compactly contained in \([0,r_0)\times I\), set
\(v_h:=\delta_h\widetilde Z_{j-1}\). Taking the difference quotient of
\eqref{eq:bootstrap-weak} gives
\[
        \iint W\Bigl[
        \beta(\widetilde Z(\cdot,t+h))(v_h)_t\Phi_t
        +A(v_h)_r\Phi_r-G_h\Phi_t-d_h\Phi
        \Bigr]\dd r\dd t=0,
\]
where \(d_h:=\delta_h f_{j-1}\) and
\[
        G_h:=\delta_h\mathfrak F_{j-1}
        -\delta_h(\beta(\widetilde Z))(\widetilde Z_{j-1})_t .
\]
The induction hypotheses imply that \(G_h\) is bounded uniformly in \(h\),
while \(d_h=Dv_h\). Testing with \(v_h\zeta^2\), where \(\zeta\) is a
relative cutoff, and using Young's inequality gives
\[
        \iint W\zeta^2|\nabla v_h|^2
        \le C\iint_{\operatorname{supp}\zeta}W(1+|v_h|^2).
\]
Since \(v_h\to\partial_t\widetilde Z_{j-1}=\widetilde Z_j\) locally uniformly, the right side is
bounded independently of \(h\). Hence
\(\widetilde Z_j\in H^1_{\loc}([0,r_0)\times I;\dd\mathfrak m)\).
Moreover, \(G_h\to\mathfrak F_j\) as \(h\to0\), with
\(\mathfrak F_j:=\partial_t\mathfrak F_{j-1}-\beta'(\widetilde Z)\widetilde Z_t(\widetilde Z_{j-1})_t\).
Passing to the limit recovers \eqref{eq:bootstrap-weak}---and hence
\eqref{eq:bootstrap-linear}---for \(\widetilde Z_j\).

\emph{Schauder at arbitrary exponent.} By
Lemma~\ref{lem:radial-regularity}, \(\widetilde Z\in C^{\alpha_0}\) up to \(r=0\)
for some \(\alpha_0\in(0,1)\). By Lemma~\ref{lem:h-boundary},
\(p_0\in C^{1,1}([0,\delta])\), so the factors \(\omega_0\) and
\(A\) lie in \(C^{0,1}\). Hence \(\mathcal A_{\widetilde Z}\in C^{0,\alpha_0}\) and
\(f_{\widetilde Z}\in L^\infty\); picking \(q>N+1\) with \(\alpha_0\le 1-(N+1)/q\),
Theorem~\ref{thm:stv-even-schauder} applied to~\eqref{eq:radial-weak}
gives \(\widetilde Z\in C^{1,\alpha_0}_{\loc}\) up to \(r=0\), and in particular \(\widetilde Z\)
is Lipschitz. Therefore \(\beta(\widetilde Z)\in C^{0,1}\) and
\(\mathcal A_{\widetilde Z}\in C^{0,1}\). Now fix any \(\tau\in(0,1)\) and choose
\(q>N+1\) with \(\tau\le 1-(N+1)/q\); since
\(\mathcal A_{\widetilde Z}\in C^{0,1}\subset C^{0,\tau}\) and \(f_{\widetilde Z}\in L^\infty\),
Theorem~\ref{thm:stv-even-schauder} reapplied yields
\(\widetilde Z\in C^{1,\tau}_{\loc}\), and in particular
\(\widetilde Z(0,\cdot)\in C^{1,\tau}_{\loc}(I)\).

\emph{Tangential induction.} Fix \(\tau\in(0,1)\) and suppose
\(\partial_t^i\widetilde Z\in C^{1,\tau}\) for \(0\le i\le j-1\); in particular
\(\widetilde Z_j=\partial_t\widetilde Z_{j-1}\in C^{0,\tau}\). The difference-quotient
estimate gives \(\widetilde Z_j\in H^1_{\loc}([0,r_0)\times I;\dd\mathfrak m)\)
together with the weak equation \eqref{eq:bootstrap-weak}. The induction
hypothesis places each factor entering \(\mathfrak F_j\), namely \(\beta^{(k)}(\widetilde Z)\)
and \(\partial_t^i\widetilde Z\) for \(i\le j\), in \(C^{0,\tau}\), so
\(\mathfrak F_j\in C^{0,\tau}\), while \(f_j=D\widetilde Z_j\in L^\infty\).
Theorem~\ref{thm:stv-even-schauder} then gives
\(\widetilde Z_j\in C^{1,\tau}\). By induction \(\partial_t^j\widetilde Z\in C^{1,\tau}\)
for every \(j\ge0\), and consequently \(\widetilde Z(0,\cdot)\in C^\infty_{\loc}(I)\).
\end{proof}

\subsection{From radial to Lagrangian coordinates}

The bootstrap of the previous subsection delivers one normal derivative
at the endpoint, in the square-root distance variable \(r=2\sqrt y\)
(recall \(a=0\)). To climb to higher normal regularity we
return to the Lagrangian variable \(y\). Expanded in non-divergence form,
equation~\eqref{eq:Z-div} reads
\[
        c_\theta p_0 Z_{yy}+(1+c_\theta)p_0' Z_y
        +p_0''Z+\partial_t(\beta(Z)Z_t)=0
\]
for \(y>0\). With \(h\) as in \eqref{eq:h-def}, the linear vanishing
\(p_0(y)=y\,h(y)\) places a single power of \(y\) in front of \(Z_{yy}\);
writing \(V:=Z_y\) and dividing by \(c_\theta h(y)\) puts the equation in
regular-singular form
\begin{equation}
\label{eq:reg-sing-ode}
        y V_y+b(y)V=F,
\end{equation}
with
\begin{equation}
\label{eq:reg-sing-coefs}
        b(y):=\frac{(1+c_\theta)p_0'(y)}{c_\theta h(y)},
        \qquad
        F:=-\frac{\partial_t(\beta(Z)Z_t)+p_0''(y)Z}{c_\theta h(y)} .
\end{equation}
At \(y=0\), \(b(0)=(1+c_\theta)/c_\theta>1\), so the singular homogeneous
solution \(y^{-b(0)}\) is non-integrable; since \(Z\) is bounded, this
mode is excluded, and \(V\) inherits the regularity of \(F\).
Lemma~\ref{lem:regular-singular-descent} below records the H\"older
version of this fact, and Proposition~\ref{prop:finite-endpoint-regularity}
iterates it in \(y\) and \(t\) to lift the endpoint regularity to
\(C^{\ell,\sigma}\) under the hypothesis \(p_0\in C^{\ell,\sigma}\).

\begin{lemma}[H\"older estimate for a regular-singular ODE]
\label{lem:regular-singular-descent}
Let \(k\ge0\), \(\alpha\in(0,1)\), \(b\in C^{k,\alpha}([0,\rho])\), and
\(b\ge b_0>0\). Suppose that \(V\in L^\infty((0,\rho)\times I)\) satisfies
\eqref{eq:reg-sing-ode} distributionally in \((0,\rho)\times I\), with
\(F\in C^{k,\alpha}([0,\rho]\times I)\).
Then
\[
        V\in C^{k,\alpha}([0,\rho]\times I),
        \qquad
        y\,\partial_y^{k+1}V\in C^{0,\alpha}([0,\rho]\times I).
\]
\end{lemma}

\begin{proof}
For \(k=0\), put
\[
        \mu(y):=\exp\int_0^y\frac{b(s)-b(0)}s\dd s .
\]
The integral is finite because \(b\in C^\alpha\), and \(\mu\) is
uniformly positive and \(C^\alpha\). Since \(F-bV\in L^\infty\), the equation gives
\(V_y\in L^\infty_{\loc}((0,\rho)\times I)\). Thus, for a.e. \(t\),
\(V(\cdot,t)\in W^{1,\infty}_{\loc}(0,\rho)\) and the equation holds
pointwise for a.e. \(y\). Multiplying this one-dimensional equation by
\(y^{b(0)-1}\mu\) gives
\[
        \partial_y\bigl(y^{b(0)}\mu V\bigr)
        =y^{b(0)-1}\mu F
\]
on \((0,\rho)\). Hence
\[
        y^{b(0)}\mu(y)V(y,t)
        =
        C(t)+\int_0^y s^{b(0)-1}\mu(s)F(s,t)\dd s ,
\]
for a.e. \(t\), where \(C(t)\) is measurable. Since \(V\in L^\infty\), the
term \(y^{-b(0)}\mu^{-1}C(t)\) cannot be locally bounded unless \(C(t)=0\).
Therefore \(V\) is represented by
\[
        V(y,t)=
        y^{-b(0)}\mu(y)^{-1}
        \int_0^y s^{b(0)-1}\mu(s)F(s,t)\dd s .
\]
Equivalently,
\[
        V(y,t)=
        \int_0^1\lambda^{b(0)-1}
        \frac{\mu(\lambda y)}{\mu(y)}F(\lambda y,t)\dd\lambda .
\]
The kernel in this Volterra formula is bounded and \(C^\alpha\) in
\(y\in[0,\rho]\), uniformly in \(\lambda\in[0,1]\), with an integrable
factor \(\lambda^{b(0)-1}\). Thus the formula defines a
\(C^{0,\alpha}([0,\rho]\times I)\) representative, with trace
\(V(0,t)=F(0,t)/b(0)\). The equation then gives
\(y V_y=F-bV\in C^{0,\alpha}([0,\rho]\times I)\).

If \(k\ge1\), then
\[
        \frac{b(y)-b(0)}{y}\in C^{k-1,\alpha}([0,\rho]),
\]
after assigning its trace at \(0\). Hence \(\mu\in C^{k,\alpha}([0,\rho])\).
The Volterra kernel
\[
        K(\lambda,y):=\lambda^{b(0)-1}
        \frac{\mu(\lambda y)}{\mu(y)}
\]
has \(C^{k,\alpha}\) norm in \(y\) bounded by an integrable function of
\(\lambda\). Differentiating the representation formula under the integral
therefore gives \(V\in C^{k,\alpha}([0,\rho]\times I)\). Finally,
differentiating \(y V_y=F-bV\) \(k\) times in \(y\) gives
\[
        y\partial_y^{k+1}V
        =\partial_y^k(F-bV)-k\partial_y^kV\in C^{0,\alpha}.
\]
\end{proof}

\begin{proposition}[Higher-order Lagrangian endpoint regularity]
\label{prop:finite-endpoint-regularity}
Suppose, in addition to the assumptions of
Theorem~\ref{thm:higher-boundary-regularity}, that
\(p_0\in C^{\ell,\sigma}([0,\delta])\) for some integer \(\ell\ge2\) and
some \(\sigma\in(0,1)\). Let \((u,m)\) be the corresponding solution and
\(\gamma\) the associated Lagrangian flow, and set
\(\gamma_L:=\gamma(0,\cdot)\). Let \(Z\) be as in \eqref{eq:Z-def} and
fix \(I\Subset(0,T)\). Then
\[
        \gamma-\gamma_L,\quad
        \partial_t(\gamma-\gamma_L)\in
        C^{\ell,\sigma}_{\loc}([0,r_0^2/4)\times I),
\]
and
\[
        (y,t)\mapsto m(\gamma(y,t),t)^\theta
        \in C^{\ell,\sigma}_{\loc}([0,r_0^2/4)\times I).
\]
Moreover, \(Z(0,\cdot)\in C^\infty_{\loc}(I)\). The symmetric statement
holds at the right endpoint \(b\), after replacing \(y\) by \(b-y\).
\end{proposition}

\begin{proof}
With \(h\) as in \eqref{eq:h-def}, the higher regularity of \(p_0\) gives
\(h\in C^{\ell-1,\sigma}([0,\delta))\) with \(h(0)>0\). We claim that, for
\(0\le i\le2\ell-1\),
\begin{equation}
\label{eq:tangential-z-bound}
        \partial_t^i Z\in C^{0,1}_{\loc}([0,r_0^2/4)\times I),
\end{equation}
and \(Z(0,\cdot)\in C^\infty_{\loc}(I)\). The substitution \(r=2\sqrt y\)
is degenerate at the axis, so this does not follow from the H\"older
bounds on \(\partial_t^i\widetilde Z\) of Proposition
\ref{prop:endpoint-schauder-bootstrap} by direct composition; the missing
rate of vanishing at the axis, \((\partial_t^i\widetilde Z)_r=O(r)\), comes from integrating the
divergence-form equation for \(\partial_t^i\widetilde Z\) against the radial weight.

Set \(\widetilde Z_i:=\partial_t^i\widetilde Z\); with the conventions
\(\mathfrak F_0=0\), \(f_0=D\widetilde Z\) introduced in the proof of
Proposition~\ref{prop:endpoint-schauder-bootstrap},
equation~\eqref{eq:bootstrap-linear} holds for all \(i\ge0\). Integrate
it in \(r\) over \((0,r)\) at fixed \(t\); the boundary term at
\(\rho=0\) vanishes since \(W(\rho)\sim\rho^{N-1}\to0\) and
\((\widetilde Z_i)_r\) is bounded by
Proposition~\ref{prop:endpoint-schauder-bootstrap}, giving
\[
        W(r)A(r)(\widetilde Z_i)_r(r,t)
        =-\int_0^rW(\rho)f_i(\rho,t)\dd\rho
        +\partial_t\!\int_0^rW(\rho)\bigl[\mathfrak F_i-\beta(\widetilde Z)(\widetilde Z_i)_t\bigr](\rho,t)\dd\rho .
\]
The first integrand is \(O(\rho^{N-1})\). For the second term, differentiate
under the integral; the differentiated integrand is
\[
        W(\rho)\,\partial_t\!\left[
        \mathfrak F_i-\beta(\widetilde Z)(\widetilde Z_i)_t
        \right](\rho,t),
\]
which is also \(O(\rho^{N-1})\), since
Proposition~\ref{prop:endpoint-schauder-bootstrap} bounds
\(\partial_t^k\widetilde Z\) for \(k\le i+2\). Both right-hand terms are therefore
\(O(r^N)\). Dividing by \(W(r)A(r)\sim r^{N-1}\) gives \((\widetilde Z_i)_r(r,t)=O(r)\),
and so
\[
        \widetilde Z_i(r,t)-\widetilde Z_i(0,t)
        =\int_0^r(\widetilde Z_i)_r(\rho,t)\dd\rho
        =O(r^2) .
\]
With \(r=2\sqrt y\), this is \(O(y)\). Moreover
\begin{equation}
\label{eq:Vi-bounded}
        V_i(y,t):=\partial_y\partial_t^iZ(y,t)
        =\frac2r(\widetilde Z_i)_r(r,t)=O(1),
\end{equation}
so \(\partial_t^iZ\) is Lipschitz in~\(y\); it is Lipschitz in~\(t\) as
well, since Proposition~\ref{prop:endpoint-schauder-bootstrap} bounds
\(\partial_t^{i+1}\widetilde Z\). Hence \(\partial_t^iZ\) is jointly Lipschitz on
\([0,r_0^2/4)\times I\), establishing \eqref{eq:tangential-z-bound}. The
smoothness of \(Z(0,\cdot)\) is immediate from
\(Z(0,t)=\widetilde Z(0,t)\in C^\infty_{\loc}(I)\), again by
Proposition~\ref{prop:endpoint-schauder-bootstrap}.

For \(0\le i\le2\ell-3\), differentiating \eqref{eq:reg-sing-ode} \(i\)
times in \(t\) shows that \(V_i\) satisfies
\begin{equation}
\label{eq:reg-sing-Vi}
        y(V_i)_y+b(y)V_i=F_i,
        \qquad F_i:=\partial_t^iF,
\end{equation}
with \(b\) and \(F\) as in \eqref{eq:reg-sing-coefs}. Since
\(c_\theta\), \(h(y)\), and \(p_0''(y)\) are independent of \(t\),
\begin{equation}
\label{eq:Fi-explicit}
        F_i=-\frac{\partial_t^{i+1}(\beta(Z)Z_t)
        +p_0''(y)\,\partial_t^iZ}{c_\theta h(y)} ,
\end{equation}
which is a polynomial in \(Z\) and its tangential derivatives
\(\partial_t^kZ\) for \(k\le i+2\), times \(y\)-coefficients built from
\(1/h(y)\) and \(p_0''(y)\). Since \(p_0\in C^{\ell,\sigma}\) and
\(h\in C^{\ell-1,\sigma}\) with \(h\ge h_0>0\), these \(y\)-coefficients
lie in \(C^{\ell-2,\sigma}([0,r_0^2/4))\); the principal coefficient
\(b\) in \eqref{eq:reg-sing-coefs} is a quotient of \(p_0'\) and \(h\), both
in \(C^{\ell-1,\sigma}\), and so itself lies in
\(C^{\ell-1,\sigma}([0,r_0^2/4))\). After decreasing \(r_0\) if
necessary, \(b\ge b_0>0\) on the endpoint interval, and the binding
regularity constraint for
Lemma~\ref{lem:regular-singular-descent} comes from \(F_i\) rather than
from \(b\).
For \(n=0\), \eqref{eq:tangential-z-bound} bounds the tangential factors
of \(F_i\) in \eqref{eq:Fi-explicit} (with \(k\le i+2\le2\ell-1\)),
placing \(F_i\) in \(C^{0,\sigma}_{\loc}([0,r_0^2/4)\times I)\); and
\(V_i\) is bounded by \eqref{eq:Vi-bounded}. Lemma
\ref{lem:regular-singular-descent} therefore gives
\begin{equation}
\label{eq:Vi-base}
        V_i,\ y(V_i)_y\in C^{0,\sigma}_{\loc}([0,r_0^2/4)\times I),
        \qquad i\le2\ell-3 .
\end{equation}
For higher normal derivatives, the strategy is to apply Lemma
\ref{lem:regular-singular-descent} to \eqref{eq:reg-sing-Vi} at order
\(n\), which requires \(F_i\in C^{n,\sigma}_{\loc}([0,r_0^2/4)\times I)\).
We prove the resulting estimate
\begin{equation}
\label{eq:Vi-n-form}
        \partial_y^nV_i,\ y\,\partial_y^{n+1}V_i
        \in C^{0,\sigma}_{\loc}([0,r_0^2/4)\times I),
        \qquad i+2(n+1)\le2\ell-1,
\end{equation}
by induction on \(n\). The 2:1 trade-off in the constraint reflects the
structure of \(F\) in \eqref{eq:reg-sing-coefs}: the term
\(\partial_t(\beta(Z)Z_t)\) makes \(F_i\) depend on tangential derivatives
of \(Z\) up to order \(i+2\), and each additional level of normal-descent
on \(V_i\) is obtained by applying
Lemma~\ref{lem:regular-singular-descent} at a tangential index higher by
\(2\)---so \(n\) levels of descent reach down to tangential order
\(i+2n\) at the bottom, requiring control of \(\partial_t^kZ\) for
\(k\le i+2(n+1)\). The cap \(2\ell-1\) is exactly the
tangential range of \eqref{eq:tangential-z-bound}.

The hypothesis at step \(n\) is that \eqref{eq:Vi-n-form} has already
been established at every earlier order \(n'<n\), for all tangential
indices \(i'\) with \(i'+2(n'+1)\le2\ell-1\); the base case \(n=0\) is
\eqref{eq:Vi-base}.

For the inductive step, it suffices to verify that
\(F_i\in C^{n,\sigma}_{\loc}\): granting this,
Lemma~\ref{lem:regular-singular-descent} applied to
\eqref{eq:reg-sing-Vi} at order \(n\) yields \eqref{eq:Vi-n-form}; the
lemma's hypothesis on the principal coefficient is met by the
\(C^{\ell-1,\sigma}\)-regularity of \(b\) noted after
\eqref{eq:Fi-explicit}. The inductive constraint
\(i+2(n+1)\le2\ell-1\) forces \(n\le\ell-2\), which is exactly the
regularity available in the \(y\)-coefficients of
\eqref{eq:Fi-explicit}; hence those coefficients admit the required
\(n\) \(y\)-derivatives.

Membership in \(C^{n,\sigma}_{\loc}\) requires control of every mixed
derivative \(\partial_y^q\partial_t^rF_i\) with \(q+r\le n\). Because
\(c_\theta\), \(h\), and \(p_0''\) are independent of \(t\),
\eqref{eq:Fi-explicit} gives \(\partial_t^rF_i=F_{i+r}\), so the
problem reduces to estimating \(\partial_y^qF_{i+r}\) with
\(q+r\le n\). When \(q=0\) no normal derivative is involved, and the
tangential factors \(\partial_t^kZ\) of \(F_{i+r}\) are bounded by
\eqref{eq:tangential-z-bound}, since \(k\le i+r+2\le2\ell-1\). When
\(q\ge1\), applying the product rule to \eqref{eq:Fi-explicit} (with
\(i\) replaced by \(i+r\)) expands \(\partial_y^qF_{i+r}\) into a sum
of terms, each consisting of a \(y\)-derivative of a coefficient (each
loss covered by \(n\le\ell-2\)) times products of factors of the form
\(\partial_y^{q_a}\partial_t^{j_a}Z\) with \(q_a\le q\) and
\(j_a\le i+r+2\). If \(q_a=0\), such a factor is again controlled by
\eqref{eq:tangential-z-bound}. If \(q_a\ge1\), it equals
\(\partial_y^{q_a-1}V_{j_a}\), and the inductive hypothesis at
\((i',n')=(j_a,q_a-1)\) places it in \(C^{0,\sigma}_{\loc}\): one has
\(q_a-1<n\), and
\[
        j_a+2q_a\le i+r+2+2q\le i+2(n+1)\le2\ell-1,
\]
where the second inequality uses \(r+2q\le2n\), itself a consequence
of \(q+r\le n\). The extreme case \((q,r)=(n,0)\) saturates these
bounds: the leading \(Z\)-factor is then
\(\partial_y^n\partial_t^{i+2}Z=\partial_y^{n-1}V_{i+2}\), placed in
\(C^{0,\sigma}_{\loc}\) by the hypothesis at \((i+2,n-1)\) under the
borderline constraint \((i+2)+2((n-1)+1)=i+2(n+1)\le2\ell-1\); the
mixed cases \(r>0\) involve strictly fewer than \(n\) normal
derivatives and so rely only on earlier steps of the same triangular
induction. Hence \(F_i\in C^{n,\sigma}_{\loc}\), and Lemma
\ref{lem:regular-singular-descent} yields \eqref{eq:Vi-n-form}.

Setting \(j:=n+1\) in the unweighted half and \(j:=n+2\) in the weighted
half, \eqref{eq:Vi-n-form} translates to: \(\partial_y^j\partial_t^iZ
\in C^{0,\sigma}_{\loc}\) whenever \(j\ge1\) and \(i+2j\le2\ell-1\), and
\(y\,\partial_y^j\partial_t^iZ\in C^{0,\sigma}_{\loc}\) whenever
\(j\ge2\) and \(i+2j-2\le2\ell-1\). The \(C^{\ell,\sigma}\) regularity
below uses only the regime \(j+i\le\ell\); we record three sub-cases:
\begin{equation}
\label{eq:zji-bulk}
        \partial_y^j\partial_t^iZ\in C^{0,\sigma}_{\loc}([0,r_0^2/4)\times I),
        \qquad 1\le j,\quad j+i\le\ell-1,
\end{equation}
and, with at least one tangential derivative,
\begin{equation}
\label{eq:zji-edge}
        \partial_y^j\partial_t^iZ\in C^{0,\sigma}_{\loc}([0,r_0^2/4)\times I),
        \qquad 1\le j,\quad i\ge1,\quad j+i\le\ell,
\end{equation}
and
\begin{equation}
\label{eq:zji-top}
        y\,\partial_y^j\partial_t^iZ\in C^{0,\sigma}_{\loc}([0,r_0^2/4)\times I),
        \qquad 1\le j,\quad j+i=\ell .
\end{equation}
Since \(Z\) is bounded above and below by positive constants, smooth
composition transfers the bounds of \eqref{eq:tangential-z-bound} and
\eqref{eq:zji-bulk}--\eqref{eq:zji-top} to
\(\gamma_y=Z^{-1/(\theta+1)}\) and to
\(\gamma_y^{-\theta}=Z^{c_\theta}\). Hence
\eqref{eq:tangential-z-bound} and \eqref{eq:zji-bulk} give
\[
        \gamma(y,t)-\gamma_L(t)
        =\int_0^y \gamma_y(s,t)\dd s
        \in C^{\ell,\sigma}_{\loc}([0,r_0^2/4)\times I),
\]
and adding \eqref{eq:zji-edge} gives
\[
        \partial_t\!\left(\gamma(y,t)-\gamma_L(t)\right)
        =\int_0^y \gamma_{yt}(s,t)\dd s
        \in C^{\ell,\sigma}_{\loc}([0,r_0^2/4)\times I).
\]
Indeed, after one \(y\)-derivative the terms to control are
\(\partial_y^j\partial_t^i\gamma_y\) with \(j+i\le\ell\) and \(i\ge1\), while
the terms without \(y\)-differentiation involve only pure tangential
derivatives of \(\gamma_y\) up to order \(\ell+1\le2\ell-1\).
For the pressure trace, \eqref{eq:mass-relation} gives
\(m(\gamma(y,t),t)^\theta=p_0(y)\gamma_y(y,t)^{-\theta}\).
Since \(p_0\in C^{\ell,\sigma}\), for derivatives of order at most
\(\ell-1\) this product is \(C^{0,\sigma}\) by
\eqref{eq:tangential-z-bound} and \eqref{eq:zji-bulk}. For derivatives
of order \(\ell\), the only terms not already controlled are those in
which all \(y\)-derivatives fall on \(\gamma_y^{-\theta}\); there
\(p_0(y)=y\,h(y)\) absorbs the factor
\(y\,\partial_y^j\partial_t^i(\gamma_y^{-\theta})\), \(j+i=\ell\),
bounded by \eqref{eq:zji-top}. Thus the pressure trace is
\(C^{\ell,\sigma}\). The proof at \(b\) is identical after replacing
\(y\) by \(b-y\).
\end{proof}

\section{Proof of the main theorem}
\label{sec:proof-main-theorem}

\begin{proof}[Proof of Theorem \ref{thm:higher-boundary-regularity}]
Fix a compact interval \(J\subset(0,T)\), and choose an open
interval \(I\Subset(0,T)\) with \(J\Subset I\). Throughout the proof we work
at the left endpoint \(a=0\); the analogous statements at \(b\) follow by
symmetry, after replacing \(y\) by \(b-y\).

At \(0\), the weak formulation \eqref{eq:radial-weak} is the model
equation \eqref{eq:radial-model} with \(B(r,t)=\beta(\widetilde Z(r,t))\). By
\eqref{eq:qy-bounds} and the definition \eqref{eq:Z-def}, \(\beta(Z)\) is
bounded above and below by positive constants; together with
\eqref{eq:radial-structure} this gives \eqref{ass:radial-weight} and
\eqref{ass:radial-coefficients}. Lemma \ref{lem:radial-regularity} and
Proposition \ref{prop:endpoint-schauder-bootstrap} then give
\[
        \widetilde Z(r,t)=Z(r^2/4,t)\in C^{1,\sigma}_{\loc}([0,r_0)\times I),
        \qquad
        \widetilde Z(0,\cdot)\in C^\infty_{\loc}(I).
\]
Since \(\widetilde Z(r,t)=Z(r^2/4,t)\) is even in \(r\), the bound
\(\widetilde Z\in C^{1,\sigma}_{\loc}\) forces \(\widetilde Z_r(0,t)=0\); applying
Proposition~\ref{prop:endpoint-schauder-bootstrap} to \(\partial_t\widetilde Z\)
likewise gives \(\partial_t\widetilde Z\in C^{1,\sigma}_{\loc}\) with
\((\partial_t\widetilde Z)_r(0,t)=0\). Set
\begin{equation}
\label{eq:Gamma-def}
        \Gamma(r,t):=\widetilde Z(r,t)^{-1/(\theta+1)}=\gamma_y(r^2/4,t),
\end{equation}
so that
\(p_0(r^2/4)\,\Gamma(r,t)^{-\theta}=m(\gamma(r^2/4,t),t)^\theta\) by
\eqref{eq:mass-relation}. Since \(\widetilde Z\) is bounded above and below by
positive constants, smooth composition gives
\begin{equation}
\label{eq:Gamma-regularity}
        \Gamma\in C^{1,\sigma}_{\loc},\qquad
        \Gamma_r(0,\cdot)=(\partial_t\Gamma)_r(0,\cdot)=0;
\end{equation}
\(\Gamma\) is then also bounded above and below by positive constants,
so the same regularity and trace conditions hold for any smooth
function of \(\Gamma\), in particular for \(\Gamma^{-\theta}\). We will repeatedly use the following elementary facts,
in which \(F,f\) denote functions of \((r,t)\) and \(y\in[0,r_0^2/4)\):
\begin{enumerate}
\item[(i)] If \(F\in C^{1,\sigma}\) and \(F_r(0,t)=0\), then
\(F(2\sqrt y,t)\in C^{0,\sigma}\).
\item[(ii)] If \(f\in C^{0,\sigma}\) and \(f(0,t)=0\), then
\(\sqrt y\,f(2\sqrt y,t)\in C^{0,\sigma}\).
\item[(iii)] If \(f\in C^{0,\sigma}\), then
\(y\,f(2\sqrt y,t)\in C^{0,\sigma}\).
\end{enumerate}
Each follows from a direct estimate on \(y\)-variations:
\(|F(2\sqrt y,t)-F(0,t)|\le Cy^{(1+\sigma)/2}\) for (i), since
\(F_r(0,t)=0\) and \(F_r\in C^\sigma\); \(|\sqrt y\,f(2\sqrt y,t)|\le
Cy^{(1+\sigma)/2}\) for (ii), since \(f(0,t)=0\); and \(|y\,f(2\sqrt
y,t)|\le Cy\) for (iii). The corresponding bounds on
\(|y^p-(y')^p|\le|y-y'|^p\) for \(p\in(0,1]\), together with the smoothness
of the substitution away from \(y=0\) and the obvious estimates in \(t\),
give joint H\"older continuity in \((y,t)\).
By \eqref{eq:Gamma-def}, \(\gamma_y(y,t)\to\Gamma(0,t)\) locally
uniformly as \(y\downarrow 0\); set \(\gamma_y(0^+,t):=\Gamma(0,t)\), so that
\(\gamma_y(0^+,\cdot)\in C^\infty_{\loc}(I)\).

For \(0<\rho<\delta\), set
\[
        G_\rho(t):=\frac1\rho\int_0^\rho\gamma(y,t)\dd y .
\]
Then \(G_\rho\to\gamma_L\) uniformly on \(I\). By \eqref{eq:gamma-Z}
and integration by parts in \(y\), using \(p_0(0)=0\),
\[
        G_\rho''
        =
        (1-c_\theta)\frac1\rho\int_0^\rho p_0'Z\dd y
        +c_\theta\frac{p_0(\rho)}{\rho}Z(\rho,\cdot).
\]
By Lemma \ref{lem:h-boundary} and the trace at \(0\), the right side converges
uniformly on \(I\) to
\[
        p_0'(0^+)\gamma_y(0^+,\cdot)^{-(\theta+1)} .
\]
Thus
\begin{equation}
\label{eq:left-acceleration}
        \gamma_L''=p_0'(0^+)\gamma_y(0^+,\cdot)^{-(\theta+1)}
\end{equation}
distributionally. Since the right side is \(C^\infty_{\loc}(I)\),
\(\gamma_L\in C^\infty_{\loc}(I)\).

In the positive phase, \eqref{eq:mass-relation} and \eqref{eq:Z-def} give
\[
        m(\gamma(y,t),t)^\theta=p_0(y)\gamma_y(y,t)^{-\theta}
        =p_0(y)Z(y,t)^{c_\theta}.
\]
For \(0<y<b\),
\begin{equation}
\label{eq:pressure-gradient}
        \partial_x(m^\theta)(\gamma(y,t),t)
        =p_0'Z+c_\theta p_0Z_y
        =\gamma_{tt}(y,t).
\end{equation}
The right side is bounded on compact subintervals of \((0,b)\times I\). Near
\(0\), write \(y=r^2/4\). Proposition
\ref{prop:endpoint-schauder-bootstrap} gives \(\widetilde Z_r\in L^\infty\)
and \(p_0(r^2/4)=O(r^2)\) by Lemma~\ref{lem:h-boundary}, so
\[
        p_0Z_y=p_0\!\left(\frac{r^2}{4}\right)\frac{2}{r}\widetilde Z_r=O(r).
\]
Thus \(\gamma_{tt}\) is bounded near \(0\), uniformly for \(t\in I\).
Since \(m\) is continuous and \(m^\theta=0\) outside the support,
\(m(\cdot,t)^\theta\) has bounded weak \(x\)-derivative on \(\R\), given
by \eqref{eq:pressure-gradient} in the positive phase and by \(0\)
elsewhere, so
\(\sup_{t\in I}\|m(\cdot,t)^\theta\|_{W^{1,\infty}(\R)}<\infty\). The
matching time bound, completing
\(p\in W^{1,\infty}_{\loc}(\R\times(0,T))\), is obtained at the end of
the proof from the chart \(P\) constructed below.

Near \(0\), by \eqref{eq:Gamma-def},
\[
        \bar{p}(y,t):=m(\gamma(y,t),t)^\theta
        =p_0(y)\,\Gamma(2\sqrt y,t)^{-\theta}
\]
is \(C^{1,\sigma}\). Indeed, by the product and chain rules,
\begin{equation}
\label{eq:p-chain-rule}
        \partial_y \bar{p}
        =p_0'(y)\,\Gamma(2\sqrt y,t)^{-\theta}
        +p_0(y)\,\frac{(\Gamma^{-\theta})_r(2\sqrt y,t)}{\sqrt y},
        \qquad
        \partial_t\bar{p}=p_0(y)\,(\Gamma^{-\theta})_t(2\sqrt y,t).
\end{equation}
With \(h\) as in \eqref{eq:h-def}, the second term in
\(\partial_y\bar{p}\) factors as
\[
        p_0(y)\frac{(\Gamma^{-\theta})_r(2\sqrt y,t)}{\sqrt y}
        =\sqrt y\,h(y)\,(\Gamma^{-\theta})_r(2\sqrt y,t);
\]
using \eqref{eq:Gamma-regularity} (which gives
\((\Gamma^{-\theta})_t,(\Gamma^{-\theta})_r\in C^{0,\sigma}_{\loc}\)),
apply (i) to \(F=\Gamma^{-\theta}\), (iii) to
\(f=(\Gamma^{-\theta})_t\), and (ii) to \(f=(\Gamma^{-\theta})_r\):
\[
        \Gamma(2\sqrt y,t)^{-\theta},\qquad
        p_0(y)(\Gamma^{-\theta})_t(2\sqrt y,t),\qquad
        \sqrt y\,h(y)(\Gamma^{-\theta})_r(2\sqrt y,t)
\]
all belong to \(C^{0,\sigma}\). Evaluating
\eqref{eq:p-chain-rule} at \(y=0\) and using \(p_0(0)=0\),
\[
        \partial_y \bar{p}(0,t)=p_0'(0^+)\Gamma(0,t)^{-\theta},
        \qquad
        \partial_t\bar{p}(0,t)=0 .
\]
Set
\[
        \widetilde\Gamma(r,t):=2\int_0^1 s\Gamma(sr,t)\dd s,\qquad
        \bar{\gamma}(y,t):=\gamma(y,t)-\gamma_L(t)=y\,\widetilde\Gamma(2\sqrt y,t).
\]
Then
\begin{equation}
\label{eq:gamma-regularity}
        \bar{\gamma}\in C^{1,\sigma}_{\loc},\qquad
        \partial_y \bar{\gamma}(y,t)=\Gamma(2\sqrt y,t),\qquad
        \partial_y \bar{\gamma}(0,t)=\Gamma(0,t)>0 .
\end{equation}
After shrinking \(\delta\), let \(Y(s,t)\) be defined by
\begin{equation}
\label{eq:Y-def}
        \bar{\gamma}(Y(s,t),t)=s,
\end{equation}
and set
\[
        P(s,t):=\bar{p}(Y(s,t),t).
\]
Since \(\partial_y \bar{\gamma}(0,\cdot)=\Gamma(0,\cdot)>0\), \(\bar{\gamma}\) admits a
\(C^{1,\sigma}\) extension across \(y=0\) to a two-sided neighborhood of
\(\{0\}\times I\) (by Whitney extension); the standard inverse function
theorem applied to \eqref{eq:Y-def} via this extension, and restricted to
\(s\ge0\), yields
\begin{equation}
\label{eq:YP-regularity}
        Y,P\in C^{1,\sigma}([0,\rho]\times I)\qquad\text{for some }\rho>0 .
\end{equation}
Since \(P(s,t)=\bar{p}(Y(s,t),t)=m(\gamma(Y(s,t),t),t)^\theta\), and since
\eqref{eq:Y-def} together with the definition of \(\bar{\gamma}\) gives
\(\gamma(Y(s,t),t)=\gamma_L(t)+s\),
\[
        p(x,t)=m(x,t)^\theta=P(x-\gamma_L(t),t)
        \quad\text{for } 0\le x-\gamma_L(t)\le\rho,\ t\in I .
\]
Together with \eqref{eq:YP-regularity}, this yields the one-sided
\(C^{1,\sigma}\) regularity of \(p\) up to the left free boundary.
The interior \(C^2\) regularity of
\(\gamma\) from Theorem \ref{thm:free-boundary-estimates}, combined with
interior Schauder estimates for the uniformly elliptic equation
\eqref{eq:flow}, then completes the proof of the
\(C^{1,\sigma}\) statement for \(p\).

Let \(y=y(x,t)\) denote the inverse of \(x=\gamma(y,t)\) in the positive
phase. Since \(u_x(\gamma(y,t),t)=-\gamma_t(y,t)\), the identities
\[
        (u_x)_x(\gamma(y,t),t)=-\frac{\gamma_{ty}(y,t)}{\gamma_y(y,t)},
        \qquad
        (u_x)_t(\gamma(y,t),t)
        =-\gamma_{tt}(y,t)+\frac{\gamma_t(y,t)\gamma_{ty}(y,t)}
        {\gamma_y(y,t)}
\]
hold in the positive phase. Differentiating
\(\gamma_y(y,t)=\Gamma(2\sqrt y,t)\) in \(t\) gives
\(\gamma_{yt}(y,t)=\Gamma_t(2\sqrt y,t)\); applying
Proposition~\ref{prop:endpoint-schauder-bootstrap} to \(\partial_t\widetilde Z\) places
\(\Gamma_t\) in \(C^{1,\sigma}_{\loc}\) with \(\Gamma_t\) even in \(r\), so (i) gives
\(\gamma_{yt}\in C^{0,\sigma}\) in \((y,t)\) up to \(y=0\). The substitution
\(y=Y(x-\gamma_L(t),t)\), \(C^{1,\sigma}\) (hence Lipschitz) in \((x,t)\) by
\eqref{eq:YP-regularity} and \(\gamma_L\in C^\infty\), then transfers
\(\gamma_{yt}\) (and likewise \(\gamma_y\)) to \(C^{0,\sigma}\) functions of
\((x,t)\) on the strip \(0\le x-\gamma_L(t)\le\rho\), \(t\in I\). Away from the endpoints, this follows from the interior
Schauder estimates for the uniformly elliptic equation \eqref{eq:flow}.
From \(\gamma_t(y,t)=\gamma_L'(t)+\partial_t\bar{\gamma}(y,t)\) and
\eqref{eq:gamma-regularity}, \(\gamma_t\in C^\sigma\), and
\(\gamma_{tt}=p_x(\gamma(y,t),t)\) belongs to \(C^\sigma\) by
\eqref{eq:pressure-gradient} and the \(C^{1,\sigma}\) regularity of \(p\) just
proved.
Therefore
\[
        u_x\in
        C^{1,\sigma}\bigl(\overline{\{m>0\}}\cap(\R\times J)\bigr).
\]
Since
\[
        u_t=\frac12u_x^2-p,
\]
and \(p\in C^{1,\sigma}\) on the same set, also
\[
        u_t\in
        C^{1,\sigma}\bigl(\overline{\{m>0\}}\cap(\R\times J)\bigr).
\]
Hence
\[
        u\in
        C^{2,\sigma}\bigl(\overline{\{m>0\}}\cap(\R\times J)\bigr).
\]
By the formula \(p(x,t)=P(x-\gamma_L(t),t)\) on
\([\gamma_L(t),\gamma_L(t)+\rho]\), with \(P\) as in \eqref{eq:YP-regularity}
and \(P(0,\cdot)=0\), the zero extension across the left free boundary belongs
to \(W^{1,\infty}\) in a full one-sided neighborhood. Indeed, for \(s=x-\gamma_L(t)>0\),
\[
        p_x=P_s(s,t),\qquad p_t=P_t(s,t)-\gamma_L'(t)P_s(s,t),
\]
while both weak derivatives vanish for \(s<0\); no boundary measure appears
because the trace of \(p\) at \(s=0\) is zero. The same argument applies at the
right free boundary. Together with the interior \(C^{1,\sigma}\) regularity,
\(p\in W^{1,\infty}_{\loc}(\R\times(0,T))\).

To upgrade this to regularity of \(u\), we apply \cite[Thm.~4.23]{CMP}:
local \(C^\beta\) regularity of the right-hand side of the
Hamilton--Jacobi equation gives \(C^{1,\beta/2}_{\loc}\) regularity of
\(u\). This conclusion fails for general viscosity solutions but holds
here because \(u\) is simultaneously a forward and backward viscosity
solution. The cited result is stated for \(\beta<1\), but its proof uses
only a local Lipschitz bound on the right-hand side and so applies
verbatim at \(\beta=1\). Hence \(p\in W^{1,\infty}_{\loc}\) yields
\(u\in C^{1,1/2}_{\loc}(\R\times(0,T))\).

Assume now that \(p_0\in C^{\ell,\sigma}([0,b])\), with one-sided derivatives at the endpoints.
Proposition \ref{prop:finite-endpoint-regularity} gives
\[
        \bar{p}(y,t):=m(\gamma(y,t),t)^\theta
\]
in \(C^{\ell,\sigma}\) on a one-sided Lagrangian endpoint chart, and
\[
        \bar{\gamma}(y,t):=\gamma(y,t)-\gamma_L(t)
\]
belongs to \(C^{\ell,\sigma}\), with
\(\partial_t\bar{\gamma}\in C^{\ell,\sigma}\). Moreover
\(\partial_y \bar{\gamma}\) is bounded below by a positive constant.
After shrinking \(\delta\), the inverse map \(Y(s,t)\) defined by
\eqref{eq:Y-def} belongs to \(C^{\ell,\sigma}([0,\rho]\times I)\)
for some \(\rho>0\), by the standard inverse function theorem
applied to a \(C^{\ell,\sigma}\) extension of \(\bar{\gamma}\) across \(y=0\). Therefore
\[
        P(s,t):=\bar{p}(Y(s,t),t)
\]
belongs to \(C^{\ell,\sigma}([0,\rho]\times I)\). Proposition
\ref{prop:finite-endpoint-regularity} also gives
\(Z(0,\cdot)\in C^\infty_{\loc}(I)\); the acceleration formula
\eqref{eq:left-acceleration} then yields \(\gamma_L\in C^\infty_{\loc}(I)\).
Define
\[
        \Omega_L^\rho:=\{(x,t):t\in I,\ 0\le x-\gamma_L(t)\le\rho\}.
\]
Since
\[
        m(x,t)^\theta=P(x-\gamma_L(t),t)
        \quad\text{in }\Omega_L^\rho,
\]
the asserted one-sided \(C^{\ell,\sigma}\) regularity of \(m^\theta\) in physical
coordinates follows near the left moving free boundary. Away from the free
boundaries, \eqref{eq:flow} is uniformly elliptic. Bootstrapping the linear
equation with coefficients
\(\theta p_0\gamma_y^{-2-\theta}\) and right side
\(p_0'\gamma_y^{-1-\theta}\), starting from the \(C^2\) regularity of
\(\gamma\) in Theorem~\ref{thm:free-boundary-estimates}, gives
\(\gamma\in C^{\ell+1,\sigma}_{\loc}\), with \(p_0'\in
C^{\ell-1,\sigma}\) as the limiting datum. Hence the change of variables
\(x=\gamma(y,t)\) is \(C^{\ell+1,\sigma}\), while
\(\bar p=p_0\gamma_y^{-\theta}\in C^{\ell,\sigma}_{\loc}\), and this transfers
to \(m^\theta\in C^{\ell,\sigma}\). Restricting to the fixed interval
\(J\) gives
\[
        m^\theta\in
        C^{\ell,\sigma}\bigl(\overline{\{m>0\}}\cap(\R\times J)\bigr).
\]
The same charts give the regularity of the value function. Near the left
boundary,
\[
        u_x(\gamma_L(t)+s,t)
        =
        -\gamma_L'(t)-\partial_t\bar{\gamma}(Y(s,t),t).
\]
The key point is that Proposition~\ref{prop:finite-endpoint-regularity}
yields the stronger conclusion \(\partial_t\bar{\gamma}\in C^{\ell,\sigma}\),
and not merely \(\bar{\gamma}\in C^{\ell,\sigma}\). Together with
\(Y\in C^{\ell,\sigma}\) and \(\gamma_L\in C^\infty_{\loc}(I)\), this yields
\(u_x\in C^{\ell,\sigma}\) up to the free boundary from the positive phase. Away from
the free boundaries, the interior regularity
\(\gamma\in C^{\ell+1,\sigma}_{\loc}\) established above gives
\(\gamma_t\in C^{\ell,\sigma}_{\loc}\), and the change of variables
\(x=\gamma(y,t)\) transfers \(u_x=-\gamma_t\circ y\) to \(C^{\ell,\sigma}\)
interiorly. The
Hamilton--Jacobi equation gives
\[
        u_t=\frac12u_x^2-m^\theta\in C^{\ell,\sigma}.
\]
Therefore
\[
        u\in
        C^{\ell+1,\sigma}\bigl(\overline{\{m>0\}}\cap(\R\times J)\bigr).
\]
Since the compact \(J\subset(0,T)\) was arbitrary, the finite-regularity
statement follows.
\end{proof}

\end{document}